\documentclass[a4paper,10pt]{article}

\usepackage{verbatim}
\usepackage{graphicx}
\usepackage{float}

\usepackage{amsmath}
\usepackage{amsthm}
\usepackage{amssymb}
\usepackage{amsbsy}

\usepackage{fouridx}
\usepackage{fullpage}

\usepackage{a4wide}
\usepackage{a4}

\usepackage{cite}
\usepackage[bottom]{footmisc}

\usepackage{hyperref}

\numberwithin{equation}{section}
\numberwithin{figure}{section}
\newtheorem{Theorem}{Theorem}
\numberwithin{Theorem}{section}

\newtheorem{Definition}[Theorem]{Definition}

\usepackage[toc,page]{appendix}

\allowdisplaybreaks

\begin{document}

\title{A Kolmogorov-Smirnov type test for two inter-dependent random variables}
\author{
Tommy Liu\footnote{University of Reading tommy.liu.academic@gmail.com}
}
\maketitle

\begin{abstract}
Consider $n$ iid random variables, where $\xi_1, \ldots, \xi_n$ are $n$ realisations of a random variable $\xi$ and $\zeta_1, \ldots, \zeta_n$ are $n$ realisations of a random variable $\zeta$. The distribution of each realisation of $\xi$, that is the distribution of \emph{one} $\xi_i$, depends on the value of the corresponding $\zeta_i$, that is the probability $P\left(\xi_i\leq x\right)=F(x,\zeta_i)$. We develop a statistical test to see if the $\xi_1, \ldots, \xi_n$ are distributed according to the distribution function $F(x,\zeta_i)$. We call this new statistical test the condition Kolmogorov-Smirnov test. 
\end{abstract}

\tableofcontents

\section{Introduction}
Random variables often occur in experiments. It is common to have $n$ observations of a quantity, which gives rise to $n$ identically and independently distributed random variables $\xi_1, \ldots, \xi_n$, where each $\xi_i$ (where $i=1,\ldots,n$) are different realisation of the same random variable $\xi$. Analysis is then done to test whether the measured $\xi_i$ follow a particular distribution. This problem was famously answered by Kolmogorov in \cite{ks_test_1}.

In this paper we consider a variant of this problem. In a physical experiment, $n$ measurements of two quantities $\xi$ and $\zeta$ are taken, resulting in $n$ iid pairs of random variables, with $\xi_1, \ldots, \xi_n$ being $n$ realisations of $\xi$ and $\zeta_1, \ldots, \zeta_n$ being $n$ realisations of $\zeta$. A distribution of $\zeta$ is not known, but the distribution of \emph{one} realisation of $\xi$, that is $\xi_i$ depends on the corresponding $\zeta_i$.  

We want to test whether the pairs of $(\xi_i,\zeta_i)$ we collected follow a particular joint distribution. The answer we offer is very similar to the original Kolmogorov-Smirnov test.

\section{Kolmogorov-Smirnov Test}
First we recall well known results about the Kolmogorov-Smirnov statistic and the Kolmogorov-Smirnov test \cite{ks_test_1}. 
Let
\begin{align*}
\xi_1, \xi_2, \ldots, \xi_n
\end{align*}
be $n$ independently and identically distributed real random variables. 
Each $\xi_i$ is distributed with  CDF $F(\cdot)$ as in 
\begin{align*}
P(\xi_i\leq x)=F(x)=\int^x_{-\infty} f(s)\,ds
\end{align*}
where $f(\cdot)$ is the PDF. 
Define a function by $F_n(\cdot)$ by 
\begin{align*}
F_n(x)=
\frac{1}{n}
\sum_{i=1}^n
\mathbf{1}_{(-\infty,x]}
(\xi_i)
\end{align*}
where $\mathbf{1}_{A}$ is the indicator function for a set $A$. 
Thus $F_n(\cdot)$ is the empirical CDF. 
Consider the distance between the real and the empirical CDF through the supremum metric on the space of real functions
\begin{align*}
D_n=
\left\Vert
F_n-F
\right\Vert_\infty
=\sup_{x\in\mathbb{R}}
\left|
\frac{1}{n}
\sum_{i=1}^n
\mathbf{1}_{(-\infty,x]}
(\xi_i)
-F(x)
\right|
\end{align*}
where $D_n$ is called the Kolmogorov-Smirnov statistic or KS statistic. 
We define what we mean by the null hypothesis. 

\begin{Definition}
Let $\xi_1,\xi_2,\ldots,\xi_n$ be $n$ real random variables.
The null hypothesis is 
that each $\xi_i$ is independently distributed with CDF $F(x)$.
\end{Definition}

\noindent We want to know how large or small $D_n$ needs to be before deciding whether to reject the null hypothesis. 
The following Theorem offers a remarkable answer to this problem.

\begin{Theorem}\label{chap_4_thm_ks_test}
Suppose the null hypothesis is true,
then the distribution of $D_n$ depends only on $n$. 
\end{Theorem}

\noindent Notice that $D_n$ is in itself a real random variable. 
The PDF and CDF of $D_n$ is a function of $n$ only, and will be the same whatever $F(\cdot)$ is.
This distribution is called the KS distribution and tables are available upto $n=100$. 
There is a Theorem which describes the asymptotic behaviour of the KS distribution. 

\begin{Theorem}
In the limit $n \longrightarrow \infty$, $\sqrt{n}D_n$ is asymptotically Kolmogorov distributed with the CDF 
\begin{equation*}
Q(x)=1-2\sum_{k=1}^\infty (-1)^{k-1}e^{-2k^2x^2}
\end{equation*}
that is to say 
\begin{equation*}
\lim_{n\longrightarrow \infty} P(\sqrt{n} D_n \leq x)=Q(x). 
\end{equation*}
\end{Theorem}

\section{Conditional Kolmogorov-Smirnov Test}
The conditional Kolmogorov-Smirnov test was first developed in the thesis \cite{tommy_thesis}. 
Let $\zeta_1,\zeta_2,\ldots,\zeta_n$ be $n$ iid real random variables. 
They are $n$ empirical observations of a random variable $\zeta$. 
Now suppose that each of the $\xi_1,\xi_2,\ldots,\xi_n$ is conditioned and dependent on the corresponding $\zeta_1,\zeta_2,\ldots,\zeta_n$.
The conditional CDF $F(\cdot,\cdot)$ is
\begin{align*}
P(\xi_i\leq x\,|\,\zeta_i)=F_{\zeta_i}(x)=\int^x_{-\infty}f(s,\zeta_i)\,ds
\end{align*}
But $\xi_1,\xi_2,\ldots,\xi_n$ are empirical measurements of the same random variable $\xi$. 
The CDF for $\xi$ is 
\begin{align*}
P(\xi\leq x)=F(x)=\int^x_{-\infty}\int^{u=+\infty}_{u=-\infty}f(s,u)m(u)\,du\,ds
\end{align*}
where $m(\cdot)$ is the PDF for $\zeta$, that is 
\begin{align*}
P(\zeta\in A)=\int_Am(s)\,ds
\end{align*}
In our context we have the problem that the random variables are not identically distributed under the null hypothesis. 
The  $\xi_1,\xi_2,\ldots,\xi_n$ and $\zeta_1,\zeta_2,\ldots,\zeta_n$ are obtained experimentally and $F_{\zeta_i}(\xi_i)$ can be  calculated but a PDF for $\zeta_i$, that is $m(\cdot)$, has no easy expression. 
We still want to perform a statistics test that is similar to the KS test even in such situations where the distribution $m(\cdot)$ of $\zeta$ is unknown. 
First we define what we call the total null hypothesis and the conditional null hypothesis. 

\begin{Definition}
Let $\xi_1, \xi_2, \ldots, \xi_n$ be $n$ empirical observations of a random variable $\xi$. 
The total null hypothesis is that
$\xi$ is distributed with the CDF $F(\cdot)$.
The conditional null is that
 each $\xi_i$ is distributed with the conditional CDF $F_{\zeta_i}(\cdot)$.
\end{Definition}

\noindent A new statistical test is developed, which is similar to the KS test. 

\begin{Theorem}
Suppose the conditional null hypothesis is true. 
Let $F_{\zeta_i}(\cdot)$ be continuous. 
Let $S_n$ be the statistic given by 
\begin{align*}
S_n=\sup_{x\in[0,1]}
\left|
\frac{1}{n}
\sum^n_{i=1}
\mathbf{1}_{[0,x]}
\left(
F_{\zeta_i}(\xi_i)
\right)
-x
\right|
\end{align*}
then $S_n$ is KS distributed. 
\end{Theorem}

\begin{proof}
Denote
\begin{align*}
Y_i=F_{\zeta_i}(\xi_i)
\end{align*}
which means 
\begin{align*}
P\left(Y_i\leq x\right)&=P\left(F_{\zeta_i}(\xi_i)\leq x\right)\\
&=P\left(\xi_i\leq F^{-1}_{\zeta_i}(x)\right)\\
&=F_{\zeta_i}\left(F^{-1}_{\zeta_i}(x)\right)\\
&=x
\end{align*}
and $0\leq Y_i\leq 1$, so $Y_i$ is uniformly distributed on $[0,1]$. 
Note that $F_{\zeta_i}(\cdot)$ is a function of one variable only. 
Let 
\begin{align*}
F_n(x)=
\frac{1}{n}
\sum_{i=1}^n
\mathbf{1}_{[0,x]}
(Y_i)
=
\frac{1}{n}
\sum_{i=1}^n
\mathbf{1}_{[0,x]}
\left(F_{\zeta_i}(\xi)\right)
\end{align*}
where $F_n(\cdot)$ is the empirical CDF of a uniformly distributed random variable, computed using $n$ observations. 
The statistic $S_n$ is the suprenum metric
\begin{align*}
S_n=\left\Vert F_n-x   \right\Vert_\infty
=\sup_{x\in[0,1]}
\left|
\frac{1}{n}
\sum^n_{i=1}
\mathbf{1}_{[0,x]}
\left(
F_{\zeta_i}(\xi_i)
\right)
-x
\right|
\end{align*}
So clearly $S_n$ is KS distributed. 
\end{proof}

\noindent We call $S_n$ the conditional KS statistic. 
Compare this to the original KS statistic, which under the assumption of the total null hypothesis can be rewritten as 
\begin{align*}
D_n=\sup_{x\in\mathbb{R}}
\left|
\frac{1}{n}
\sum_{i=1}^n
\mathbf{1}_{(-\infty,x]}
(\xi_i)
-F(x)
\right|
=\sup_{x\in[0,1]}
\left|
\frac{1}{n}
\sum^n_{i=1}
\mathbf{1}_{[0,x]}
\left(
F(\xi_i)
\right)
-x
\right|
\end{align*}
When both the total and conditional hypothesis are true $D_n$ and $S_n$ are KS distributed, that is 
\begin{align*}
P\left(D_n\in A \right)=P\left(S_n\in A \right)
\quad \text{and} \quad 
P\left(D_n \leq x\right)=P\left(S_n\leq x\right)
\end{align*}
The subtlety here is that
$D_n$ and $S_n$ are different objects, yet they have the same distribution. 
$D_n$ is KS distributed under the total null hypothesis, 
whereas $S_n$ is KS distributed under the conditional null hypothesis. 
This can be explained in another way.
We have $n$ experimental observations of a random variable $\xi$ denoted by $\xi_1,\xi_2, \ldots, \xi_n$ and each are conditioned on observations of another random variable $\zeta$ denoted by $\zeta_1, \zeta_2, \ldots, \zeta_n$. 
The $D_n$ is KS distributed if the random variable $\xi$ is distributed by CDF $F(\cdot)$, 
but the $S_n$ is KS distributed if each $\xi_i$ is conditionally distributed by the CDF $F_{\zeta_i}(\cdot)$. 

\section{Conclusion}
An example of the conditional Kolmogorov-Smirnov test reliably working with experimental data is studied in the paper \cite{tommy_paper_B}, which is also the first time of it being implemented.

\addcontentsline{toc}{section}{References}
\bibliographystyle{ieeetr}
\bibliography{paper_reportreferences}

\end{document}